\documentclass[12pt,a4paper]{article}
\usepackage{amsfonts}
\usepackage{latexsym}
\usepackage{amsthm}
\usepackage{amsmath}
\usepackage{amssymb}
\usepackage{array}
\usepackage{enumerate}

\newtheorem{theorem}{Theorem}[section]
\newtheorem{corollary}[theorem]{Corollary}
\newtheorem{lemma}[theorem]{Lemma}
\newtheorem{proposition}[theorem]{Proposition}

\newtheorem{example}[theorem]{Example}

\begin{document}

\title{\bf On the minimum degree of the power graph of a finite cyclic group}
\author{Ramesh Prasad Panda \and Kamal Lochan Patra \and Binod Kumar Sahoo}

\maketitle

\begin{abstract}
The power graph $\mathcal{P}(G)$ of a finite group $G$ is the simple undirected graph whose vertex set is $G$, in which two distinct vertices are adjacent if one of them is an integral power of the other. For an integer $n\geq 2$, let $C_n$ denote the cyclic group of order $n$ and let $r$ be the number of distinct prime divisors of $n$. The minimum degree $\delta(\mathcal{P}(C_n))$ of $\mathcal{P}(C_n)$ is known for $r\in\{1,2\}$, see \cite{PK-CA}. For $r\geq 3$, under certain conditions involving the prime divisors of $n$, we identify at most $r-1$ vertices such that $\delta(\mathcal{P}(C_n))$ is equal to the degree of at least one of these vertices. If $r=3$ or if $n$ is a product of distinct primes, we are able to identify two such vertices without any condition on the prime divisors of $n$.

\vspace{10pt}


\vspace{5pt}

\end{abstract}

\section{Introduction}

Let $\Gamma$ be a simple graph with at least two vertices. The {\it edge connectivity} $\kappa'(\Gamma)$ of $\Gamma$ is the minimum number of edges whose deletion from $\Gamma$ gives a disconnected subgraph of $\Gamma$. The {\it vertex connectivity} $\kappa(\Gamma)$ of $\Gamma$ is the minimum number of vertices which need to be removed from $\Gamma$ so that the induced subgraph of $\Gamma$ on the remaining vertices is disconnected or has only one vertex. The latter case arises only when $\Gamma$ is a complete graph. The {\it minimum degree} of $\Gamma$, denoted by $\delta(\Gamma)$, is the minimum of the degrees of vertices of $\Gamma$. The study of vertex/edge connectivity is an interesting problem in graph theory.  It is known that $\kappa(\Gamma)\leq \kappa'(\Gamma)\leq \delta(\Gamma)$, and $\kappa'(\Gamma)=\delta(\Gamma)$ if the diameter of $\Gamma$ is at most $2$, see Theorem 4.1.9 and Exercise 4.1.25 in \cite{west}.

\subsection{Power graph}

The notion of directed power graph of a group was introduced in \cite{kel-2000}, which was further extended to semigroups in \cite{kel-2001, kel-2002}. Then the undirected power graph of a semigroup, in particular, of a group was defined in \cite{CGS-2009}. Many researchers have investigated both directed and undirected power graphs of groups from different view points. More on these graphs can be found in the survey paper \cite{AKC-2013} and the references therein.

Let $G$ be a finite group. The {\it power graph} of $G$, denoted by $\mathcal{P}(G)$, is the simple undirected graph with vertex set $G$, in which two distinct vertices are adjacent if one of them can be written as an integral power of the other. Since $G$ is finite, the identity element of $G$ is adjacent to all other vertices. So $\mathcal{P}(G)$ is a connected graph and its diameter is at most 2.

By \cite[Theorem 2.12]{CGS-2009}, $\mathcal{P}(G)$ is a complete graph if and only if $G$ is a cyclic group of prime power order. It is proved in \cite[Theorem 1.3]{cur-2014} and \cite[Corollary 3.4]{cur-2016} respectively that, among all finite groups of a given order, the cyclic group of that order has the maximum number of edges and has the largest clique in its power graph. By \cite[Theorem 5]{dooser} and \cite[Corollary 2.5]{FMW}, the power graph of a finite group is perfect, in particular, the clique number and the chromatic number coincide. Explicit formula for the clique number of the power graph of a finite cyclic group is given in \cite[Theorem 2]{mir} and \cite[Theorem 7]{dooser}. The full automorphism group of the power graph of a finite group is described in \cite[Theorem 2.2]{FMW-16}.

For a positive integer $n$, let $C_n$ denote the finite cyclic group of order $n$. The vertex connectivity of $\mathcal{P}(C_n)$ is studied in \cite{CP-ADM, CPS, CPS-1, PK-JAA} and the exact value of $\kappa(\mathcal{P}(C_n))$ is obtained in the following cases: (i) $n$ is a product of distinct primes, (ii) $n$ is divisible by at most three distinct primes, (iii) $n$ is divisible by the square of its largest prime factor, and (iv) the smallest prime divisor of $n$ is greater than or equal to the number of distinct prime divisors of $n$. The above articles also provide some sharp upper bounds for $\kappa(\mathcal{P}(C_n))$. It is proved in \cite[Theorem 6.7]{PK-CA} that the vertex connectivity and the minimum degree of $\mathcal{P}(C_n)$ coincide if and only if either $n$ is a prime power or $n$ is twice of an odd prime power. For these values of $n$, the relation $\kappa(\mathcal{P}(C_n))\leq \kappa'(\mathcal{P}(C_n))\leq \delta(\mathcal{P}(C_n))$ implies that the vertex connectivity and the edge connectivity of $\mathcal{P}(C_n)$ are equal. Since the diameter of $\mathcal{P}(C_n)$ is at most $2$, the edge connectivity and the minimum degree of $\mathcal{P}(C_n)$ coincide for every $n$. Thus, in order to determine the edge connectivity of $\mathcal{P}(C_n)$, it is enough to find the minimum degree of $\mathcal{P}(C_n)$.

\subsection{Minimum degree of $\mathcal{P}(C_n)$}

Throughout the paper, we shall identify $C_n$ with $\mathbb{Z}_n=\{0,1,2,\ldots,n-1\}$, the group of integers modulo $n$. The degree of a vertex $a\in C_n$ is denoted by $\deg(a)$. By \cite[Lemma 3.4]{MRS-JAA} (also see \cite[Lemma 2.7]{cur-2014}), we have the following formula for $\deg(a)$:
\begin{equation}\label{eqn-2}
\deg({a})=\frac{n}{b} +\underset{d|b,\; d\neq b}\sum \phi \left( \frac{n}{d} \right) -1=\frac{n}{b} +\underset{d|b}\sum \phi \left( \frac{n}{d} \right) - \phi \left( \frac{n}{b} \right) -1,
\end{equation}
where $\phi$ is the Euler's totient function and $b$ is the greatest common divisor of $a$ and $n$. If $a=0$ or $a$ is a generator of $C_n$, then $\deg(a)=n-1$.

To determine $\delta(\mathcal{P}(C_n))$, our objective will be to identify a vertex of $\mathcal{P}(C_n)$ having minimum degree and then the degree of that vertex can be calculated using (\ref{eqn-2}). The formula (\ref{eqn-2}) implies that $\deg(a)=\deg(b)$. Thus the degree of a given non-zero vertex of $\mathcal{P}(C_n)$ is equal to the degree of some element of $C_n$ which is a divisor of $n$. Therefore, in order to identify a vertex of $\mathcal{P}(C_n)$ of minimum degree, we need to compare the degrees of all possible vertices which are divisors of $n$.

If $n$ is a prime power, then $\mathcal{P}(C_n)$ is a complete graph and so $\delta(\mathcal{P}(C_n))=n-1=\deg(a)$ for every vertex $a\in C_n$. If $n>1$ is not a prime power, then $\mathcal{P}(C_n)$ is not a complete graph and so $\delta(\mathcal{P}(C_n))<n-1$. Hence the minimum degree of $\mathcal{P}(C_n)$ will be equal to the degree of a vertex which is a proper\footnote{A positive integer $a$ is called a {\it proper divisor} of $n$ if $a$ divides $n$ and $a\notin\{1,n\}$.} divisor of $n$. For certain values of $n$, a vertex of $\mathcal{P}(C_n)$ of minimum degree was obtained in \cite[Theorem 4.6]{PK-CA} which we mention below.

\begin{proposition}\cite{PK-CA}\label{mdv}
Let $p_1,p_2,p_3,p_4$ be prime numbers with $p_1<p_2<p_3<p_4$. Then the following hold:
\begin{enumerate}[\rm(i)]
\item If $n=p_1^{\alpha_1} p_2^{\alpha_2}$ for some positive integers $\alpha_1, \alpha_2$, then $\delta(\mathcal{P}(C_n)) = \deg\left({p_2^{\alpha_2}}\right)$.
\item If $n=p_1 p_2 p_3$, then $\delta(\mathcal{P}(C_n)) = \deg({p_3})$.
\item Let $n = p_1 p_2 p_3p_4$. If $n$ is odd or $p_4 \geq p_3 + \displaystyle \frac{2(p_3-1)}{p_2-1}$, then $\delta(\mathcal{P}(C_n)) = \deg({p_4})$, otherwise, $\delta(\mathcal{P}(C_n)) = \deg({p_3p_4})$.
\end{enumerate}
\end{proposition}

In this paper, we generalize the results stated in Proposition \ref{mdv} to several other values of $n$. In view of Proposition \ref{mdv}(i), if necessary, we may assume that $n$ is divisible by at least three distinct prime numbers.

The following theorem is proved in Section \ref{distinct-primes} for the minimum degree of $\mathcal{P}(C_n)$ when $n$ is a product of distinct prime numbers.

\begin{theorem}\label{mindeg.main}
Let $n=p_1p_2\cdots p_r$, where $r \geq 3$ and $p_1,p_2,\ldots, p_r$ are prime numbers with $p_1<p_2<\cdots <p_r$. Then
$$\delta(\mathcal{P}(C_n)) = \min\{\deg(p_{r-1}p_r), \deg(p_r)\}.$$
Further, $\delta(\mathcal{P}(C_n)) = \deg(p_r)$ if and only if $\phi(p_r) \geq \left( \dfrac{p_1p_2\cdots p_{r-2}}{\phi(p_1p_2\cdots p_{r-2})} - 1\right) \phi(p_{r-1})$.
In particular, if $\phi(p_{r}) \geq (r-2) \phi(p_{r-1})$, then $\delta(\mathcal{P}(C_n)) = \deg({p_r})$.
\end{theorem}

For general $n$, under certain conditions involving its prime divisors, the following theorem is proved in Section \ref{general-n} on the minimum degree of $\mathcal{P}(C_n)$.

\begin{theorem}\label{thm.mindeg}
Let $n=p_1^{\alpha_1}p_2^{\alpha_2} \cdots p_r^{\alpha_r}$, where $r \geq 2$, $\alpha_1,\alpha_2,\ldots, \alpha_r$ are positive integers and $p_1,p_2,\ldots, p_r$ are prime numbers with $p_1<p_2<\cdots <p_r$. Suppose that any of the following two conditions holds:
\begin{enumerate}[\rm(i)]
	\item $2\phi(p_{1}p_2 \cdots p_{r}) \geq p_1p_2\cdots p_r$,
	\item $\phi(p_{i+1}) \geq r \phi(p_i)$ for each $i\in\{1,2,\ldots, r-1\}$.
\end{enumerate}
If $t\in\{2,3,\ldots, r\}$ is the largest integer such that $\alpha_t\geq \alpha_j$ for $2\leq j\leq r$, then
$$\delta(\mathcal{P}(C_n)) = \min\{\deg\left(p_s^{\alpha_s}\right):t \leq s \leq r\}.$$
\end{theorem}

As an application of Theorem \ref{thm.mindeg}, we prove the following corollary in Section \ref{general-n} which can be used to determine $\delta(\mathcal{P}(C_n))$ for many values of $n$.

\begin{corollary}\label{exact.mindeg}
Let $n=p_1^{\alpha_1}p_2^{\alpha_2} \cdots p_r^{\alpha_r}$, where $r \geq 2$, $\alpha_1,\alpha_2,\ldots, \alpha_r$ are positive integers and $p_1<p_2<\cdots <p_r$ are prime numbers. Suppose that any of the following two conditions holds:
\begin{enumerate}[\rm(i)]
\item $p_1\geq  r+1$ and $p_r > rp_{r-1}$,	
\item $p_{i+1} > rp_{i}$ for each $i\in\{1,2,\ldots, r-1\}$.	
\end{enumerate}
Then $\delta(\mathcal{P}(C_n)) = \deg(p_r^{\alpha_r})$.
\end{corollary}

For $r=3$, the following theorem is proved in Section \ref{three-primes} which shows that the conclusion of Theorem \ref{thm.mindeg} holds good without any condition involving the prime divisors of $n$.

\begin{theorem}\label{mindeg.3prime}
Let $n=p_1^{\alpha_1}p_2^{\alpha_2}p_3^{\alpha_3}$, where $\alpha_1,\alpha_2,\alpha_3$ are positive integers and $p_1,p_2,p_3$ are prime numbers with $p_1<p_2<p_3$. Then
\begin{align*}
\delta(\mathcal{P}(C_n)) = \min\{\deg\left({p_2^{\alpha_2}}\right),\;\deg\left({p_3^{\alpha_3}}\right)\}.
\end{align*}
\end{theorem}

\subsection{Remark}
We remark that Proposition \ref{mdv} can be obtained from Theorems \ref{mindeg.main} and \ref{thm.mindeg}.
\begin{enumerate}[$\bullet$]
\item If $r=3$ and $n=p_1p_2p_3$, then $\phi(p_r)=\phi(p_3)> \phi(p_2)=(r-2)\phi(p_{r-1})$ and so Proposition \ref{mdv}(ii) follows from the last part of Theorem \ref{mindeg.main}.

\item Suppose that $r=4$ and $n=p_1p_2p_3p_4$. If $n$ is odd, then $p_1\geq 3$ and so $2\phi(p_1p_2)> p_1p_2$ by Lemma \ref{prime.ineq}(ii). Then $\phi(p_4)>\phi(p_3) >\left( \dfrac{p_1p_2}{\phi(p_1p_2)} - 1\right) \phi(p_{3})$. If $n$ is even, then $p_1=2$ and so $1+\left( \dfrac{p_1p_2}{\phi(p_1p_2)} - 1 \right) \phi(p_3)=p_3 + \displaystyle \frac{2(p_3-1)}{p_2-1}$. In this case, $p_4\geq 1+ \left( \dfrac{p_1p_2}{\phi(p_1p_2)} - 1 \right) \phi(p_3)$ if and only if $p_4\geq p_3 + \displaystyle \frac{2(p_3-1)}{p_2-1}$. Then it follows that Proposition \ref{mdv}(iii) can be obtained from Theorem \ref{mindeg.main}.

\item Finally, suppose that $n=p_1^{\alpha_1}p_2^{\alpha_2}$. If $p_1\geq 3$, then $2\phi(p_1p_2)> p_1p_2$ by Lemma \ref{prime.ineq}(ii). If $p_1=2$, then $\phi(p_{2}) \geq 2=2 \phi(p_1)$. Thus condition (i) or (ii) of Theorem \ref{thm.mindeg} is satisfied and hence Proposition \ref{mdv}(i) follows from Theorem \ref{thm.mindeg}.
\end{enumerate}

\section{Preliminaries}

Recall that $\phi$ is a multiplicative function, that is, $\phi(ab)=\phi(a)\phi(b)$ for any two positive integers $a,b$ which are relatively primes. We have $\phi(p^k)=p^{k-1}\phi(p)$ for any prime number $p$ and positive integer $k$. Also,
\begin{equation}\label{eqn-2-1}
\underset{d|m}\sum \phi(d) = m
\end{equation}
for every positive integer $m$. We need the following two inequalities: the first one can be found in \cite[Lemma 3.1]{CPS-2} and the second one was proved in \cite{CPS} while proving Corollary 1.4.

\begin{lemma}\cite{CPS,CPS-2}\label{prime.ineq}
Let $p_1< p_2< \cdots < p_t$ be prime numbers. Then the following hold:
\begin{enumerate}[\rm(i)]
\item $(t+1)\phi(p_1p_2 \cdots p_t) \geq p_1p_2 \cdots p_t$, with equality if and only if $(t, p_1)=(1,2)$ or $(t, p_1,p_2)=(2,2,3)$.
\item If $p_1 \geq t+1$, then  $2\phi(p_1 p_2 \cdots p_t) \geq p_1 p_2 \cdots p_t$, with equality if and only if $t=1$ and $p_1=2$.
\end{enumerate}		
\end{lemma}

Certain inequalities involving degree of vertices of $\mathcal{P}(C_n)$ were proved in \cite[Proposition 4.5]{PK-CA}. From the proof of these inequalities, it can be seen that those inequalities are in fact strict and we have stated them accordingly in the following proposition.

\begin{proposition}\cite{PK-CA}\label{degcompare}
Let $n=p_1^{\alpha_1}p_2^{\alpha_2} \cdots p_r^{\alpha_r}$, where $r \geq 2$, $\alpha_1,\alpha_2,\ldots, \alpha_r$ are positive integers and $p_1,p_2,\ldots, p_r$ are prime numbers with $p_1<p_2<\cdots <p_r$. Then the following strict inequalities hold in $\mathcal{P}(C_n):$
\begin{enumerate}[\rm(i)]
\item $\deg\left({p_1^{\alpha_1}}\right) > \deg\left({p_r^{\alpha_r}}\right)$.
\item $\deg\left({p_i^{\gamma}}\right) > \deg\left({p_i^{\beta}}\right)$ for $1 \leq i \leq r$ and $1 \leq \gamma < \beta \leq \alpha_i$.
\item $\deg\left({p_i^\beta}\right) > \deg\left({p_j^\beta}\right)$ for $1 \leq i < j \leq r$ and $1 \leq \beta \leq \min \{\alpha_i, \alpha_j\}$.
\item $\deg\left({p_1^{\beta_1}p_2^{\beta_2}\cdots p_r^{\beta_r}}\right) > \deg\left({p_2^{\beta_2}\cdots p_r^{\beta_r}}\right)$, where $1 \leq \beta_i \leq \alpha_i$ for each $i\in\{1,2,\ldots, r\}$.
\end{enumerate}
\end{proposition}

We need the strict inequality (\ref{eqn-5}) stated in the following lemma while proving Corollaries \ref{exact.mindeg} and \ref{coro}.

\begin{lemma}
Let $n=p_1^{\alpha_1}p_2^{\alpha_2} \cdots p_r^{\alpha_r}$, where $r \geq 2$, $\alpha_1,\alpha_2,\ldots, \alpha_r$ are positive integers and $p_1,p_2,\ldots, p_r$ are prime numbers with $p_1<p_2<\cdots <p_r$. For $i\in\{1,2,\ldots, r-1\}$, the following strict inequality holds in $\mathcal{P}(C_n):$
\begin{equation}\label{eqn-5}
\deg\left({p_i^{\alpha_i}}\right) - \deg\left({p_r^{\alpha_r}}\right) > p_i^{\alpha_i-1} \left[(p_r-1)\phi\left(\frac{n}{p_i^{\alpha_i}p_r^{\alpha_r}}\right) - \frac{n}{p_i^{\alpha_i -1}p_r^{\alpha_r}}\right].
\end{equation}
\end{lemma}

\begin{proof}
This follows from the proof of \cite[Proposition 4.5(i)]{PK-CA}, in which we replace the subscript $1$ by $i$ and take $m=\dfrac{n}{p_i^{\alpha_i}p_r^{\alpha_r}}$. We note that the first inequality in the proof of \cite[Proposition 4.5(i)]{PK-CA} is strict.
\end{proof}

\begin{lemma}\label{pp1}
Let $n=p_1p_2\cdots p_r$, where $p_1, p_2, \ldots, p_r$ are prime distinct numbers and let $a,b$ be two distinct proper divisors of $n$ such that $\dfrac{a}{b}=\dfrac{p_k}{p_l}$ for some $k,l\in\{1,2,\ldots ,r\}$. If $a<b$, then $\deg(a)>\deg(b)$.
\end{lemma}

\begin{proof}
Note that both $a$ and $b$ have the same number of prime divisors, say $s$. Since $a<b$, we have $p_k<p_l$ and $1\leq s\leq r-1$. The lemma follows from Proposition \ref{degcompare}(iii) if $s=1$. Assume that $s\geq 2$. We have
\begin{align*}
\underset{d | a,\: d < a}\sum \phi \left(\frac{n}{d}\right)-\underset{d | b,\: d < b}\sum \phi \left( \frac{n}{d}\right)& = \underset{d | a,\: d < a, \: p_{k} | d}\sum \phi \left( \frac{n}{d}\right)- \underset{d | b,\: d <b, \: p_{l} | d}\sum \phi \left( \frac{n}{d}\right)\\
& = \underset{u \vert \frac{a}{p_{k}},\: u < \frac{a}{p_{k}}}\sum \phi \left( \frac{n}{up_{k}}\right)  - \underset{v | \frac{b}{p_{l}},\: v < \frac{b}{p_{l}}}\sum \phi \left( \frac{n}{vp_{l}}\right)\\
& = \underset{u | \frac{a}{p_{k}},\: u < \frac{a}{p_{k}}}\sum \left[\phi \left( \frac{n}{up_{k}}\right)  - \phi \left( \frac{n}{up_{l}}\right) \right]\\
& = \left[\phi (p_{l})  - \phi (p_{k}) \right] \left[ \underset{u | \frac{a}{p_{k}},\: u < \frac{a}{p_{k}}}\sum \phi \left( \frac{n}{up_{k}p_{l}}\right)\right].
\end{align*}
The second last equality in the above holds as $\dfrac{a}{p_k}=\dfrac{b}{p_l}$. Using the formula (\ref{eqn-2}), we then get
\begin{align*}
\deg(a) - \deg(b)& = \frac{n}{a}-\frac{n}{b} + \underset{d | a,\: d < a}\sum \phi \left(\frac{n}{d}\right)-\underset{d | b,\: d < b}\sum \phi \left( \frac{n}{d}\right)\\
& > \underset{d | a,\: d < a}\sum \phi \left(\frac{n}{d}\right)-\underset{d | b,\: d < b}\sum \phi \left( \frac{n}{d}\right)\\
& = \left[\phi (p_{l})  - \phi (p_{k}) \right] \left[ \underset{u | \frac{a}{p_{k}},\: u < \frac{a}{p_{k}}}\sum \phi \left( \frac{n}{up_{k}p_{l}}\right)\right].
\end{align*}
Since $p_l >p_k$, we have $\phi (p_{l})  - \phi (p_{k})>0$ and it follows from the above that $\deg(a)>\deg(b)$.
\end{proof}

\begin{lemma}\label{phi.sum}
Let $x=p_1^{\alpha_1}p_2^{\alpha_2} \cdots p_r^{\alpha_r}$, where $\alpha_1,\alpha_2,\ldots, \alpha_r$ are positive integers and $p_1,p_2,\ldots, p_r$ are prime numbers. If $y = p_1^{\beta_1}p_2^{\beta_2}\cdots p_r^{\beta_r}$, where $0\leq \beta_i \leq \alpha_i$ for $1 \leq i \leq r$, then
\begin{align*}
\sum_{d | y} \phi \left( \frac{x}{d}\right) = \left( \prod\limits_{\substack{1\leq i \leq r\\\alpha_i = \beta_i}} p_i^{\alpha_i}  \right) \left(\prod\limits_{\substack{1\leq i \leq r\\\alpha_i > \beta_i}} \left(p_i^{\alpha_i} - p_i^{\alpha_i - \beta_i -1}\right) \right).
\end{align*}
\end{lemma}

\begin{proof}
If $\alpha_i=\beta_i$ for all $i\in\{1,2,\ldots, r\}$, then $y=x$ and the result follows from the fact that $\displaystyle\sum\limits_{d |x} \phi \left( \frac{x}{d}\right) = \displaystyle\sum\limits_{d |x} \phi \left( d\right)=x$. So assume that $\beta_i < \alpha_i$ for at least one $i\in\{1,2,\ldots, r\}$. For each $j\in\{1,2,\ldots, r\}$, observe that
\begin{equation}\label{eqn-3}
\underset{0\leq \gamma_j \leq \beta_j}\sum \phi\left(p_j^{\alpha_j-\gamma_j}\right)= p_j^{\alpha_j}\;  \mbox{ or }\; p_j^{\alpha_j} - p_j^{\alpha_j - \beta_j -1} 
\end{equation}
according as $\alpha_j =\beta_j$ or $\alpha_j >\beta_j$. We have
\begin{align*}
\sum_{d |y} \phi \left( \frac{x}{d}\right) = \sum_{\substack{1 \leq j \leq r \\ 0 \leq  \gamma_j \leq \beta_j }} \phi \left( p_1^{\alpha_1 - \gamma_1} p_2^{\alpha_2 - \gamma_2}\cdots p_r^{\alpha_r - \gamma_r} \right) = \prod_{1\leq j \leq r}  \left( \sum_{0 \leq  \gamma_j \leq \beta_j} \phi\left(p_j^{\alpha_j-\gamma_j}\right)\right),
\end{align*}
and consequently the lemma follows using (\ref{eqn-3}).
\end{proof}

\section{Proof of Theorem \ref{mindeg.main}}\label{distinct-primes}

In this section, we take $n=p_1p_2\cdots p_r$, where $r \geq 3$ and $p_1,p_2,\ldots, p_r$ are prime numbers with $p_1<p_2<\cdots <p_r$.

\begin{lemma}\label{pp2}
$\delta(\mathcal{P}(C_n)) = \min \{\deg(p_sp_{s+1}\cdots p_r): 2 \leq s \leq r\}$.
\end{lemma}

\begin{proof}
Let $\{k_1,k_2,\ldots,k_t\}$ be a proper subset of $\{1,2,\cdots,r\}$ with $k_1<k_2<\cdots <k_t$. Applying Lemma \ref{pp1} repeatedly, we get
\begin{align*}
\deg({p_{k_1}\cdots p_{k_{t-1}} p_{k_t}}) \geq \deg({p_{k_1}\cdots p_{k_{t-1}} p_r})\geq \cdots \geq \deg({p_{{r-t+1}}\cdots p_{r-1} p_r}).
\end{align*}
Since $\delta\left(\mathcal{P}(C_n)\right)$ is equal to the degree of a vertex which is a proper divisor of $n$, it follows from the above that $\delta(\mathcal{P}(C_n)) = \min \{\deg(p_sp_{s+1}\cdots p_r): 2 \leq s \leq r\}$.
\end{proof}

\begin{lemma}\label{ineq.cond}
Let $3 \leq s \leq r$. Then $\deg({p_{s-1}p_{s}\cdots p_r}) \geq \deg({p_{s}p_{s+1}\cdots p_r})$ if and only if
	\begin{align}\label{prod.p.ineq}
	p_{s}p_{s+1}\cdots p_r \geq \left( \frac{p_1p_2\cdots p_{s-2}}{\phi(p_1p_2\cdots p_{s-2})} - 1 \right) \phi(p_{s-1}) +1.
	\end{align}
Further, $\deg({p_{s-1}p_{s}\cdots p_r}) =\deg({p_{s}p_{s+1}\cdots p_r})$ if and only if equality holds in (\ref{prod.p.ineq}).
\end{lemma}

\begin{proof}
We have $\dfrac{n}{p_{s-1}p_{s}\cdots p_r} - \dfrac{n}{p_{s}p_{s+1}\cdots p_r}=- p_1p_2\cdots p_{s-2} \phi(p_{s-1})$. Using (\ref{eqn-2-1}), an easy calculation gives that
$$\underset{d | (p_{s-1}p_{s}\cdots p_r)}\sum \phi \left( \frac{n}{d}\right) - \underset{d | (p_{s}p_{s+1}\cdots p_r)}\sum \phi \left( \frac{n}{d}\right)=\phi(p_1p_2\cdots p_{s-2}) p_{s}p_{s+1}\cdots p_r.$$
We also have
$$\phi\left(\dfrac{n}{p_{s}p_{s+1}\cdots p_r} \right)  - \phi\left(\dfrac{n}{p_{s-1}p_{s}\cdots p_r} \right)=\phi(p_1p_2\cdots p_{s-2} )(\phi(p_{s-1})-1).$$
Then, using the degree formula (\ref{eqn-2}), it follows that
\begin{align*}
 &	\deg  ({p_{s-1}p_{s}\cdots p_r}) - \deg({p_{s}\cdots p_r})\\
 & = - p_1p_2\cdots p_{s-2} \phi(p_{s-1})+ \phi(p_1p_2\cdots p_{s-2}) p_{s}p_{s+1}\cdots p_r + \phi(p_1p_2\cdots p_{s-2} )(\phi(p_{s-1})-1)\\
 & = \phi(p_1p_2\cdots p_{s-2}) p_{s}p_{s+1}\cdots p_r  - \left[p_1p_2\cdots p_{s-2} - \phi(p_1p_2\cdots p_{s-2} )\right] \phi(p_{s-1})- \phi(p_1p_2\cdots p_{s-2} ).
\end{align*}
Now it can be seen that $\deg({p_{s-1}p_{s}\cdots p_r}) \geq \deg({p_{s}p_{s+1}\cdots p_r})$ if and only if inequality (\ref{prod.p.ineq}) holds.
Also, $\deg({p_{s-1}p_{s}\cdots p_r}) =\deg({p_{s}p_{s+1}\cdots p_r})$ if and only if equality holds in (\ref{prod.p.ineq}).
\end{proof}

\begin{lemma}\label{higher-lower}
Let $3\leq s\leq r$. If $\deg (p_{s-1}p_{s}\cdots p_r) \geq \deg(p_{s}p_{s+1}\cdots p_r)$, then
$$\deg (p_{s-2}p_{s-1}\cdots p_r) > \deg(p_{s-1}p_{s}\cdots p_r).$$
\end{lemma}

\begin{proof}
If $s=3$, then $\deg(p_1p_2\cdots p_r)=\deg(n)=\deg(0)=n-1>\dfrac{n\phi(p_1)}{p_1}=\deg(p_2p_3\cdots p_r)$ and the lemma follows.
Assume that $4\leq s\leq r$. Observe that the inequality (\ref{prod.p.ineq}) in the statement of Lemma \ref{ineq.cond} is equivalent to the following inequality:
\begin{align}\label{eq.mindeg2}
	\phi(p_1p_2\cdots p_{s-2}) (p_{s}p_{s+1}\cdots p_{r}-1) \geq \phi(p_{s-1}) \left[p_1p_2\cdots p_{s-2} - \phi(p_1p_2\cdots p_{s-2})\right].
	\end{align}	

Since $\deg (p_{s-1}p_{s}\cdots p_r) \geq \deg(p_{s}p_{s+1}\cdots p_r)$ by the given hypothesis, Lemma \ref{ineq.cond} then implies that the inequality (\ref{eq.mindeg2}) holds. We need to show that
$$\deg (p_{s-2}p_{s-1}\cdots p_r) > \deg(p_{s-1}p_{s}\cdots p_r).$$
By Lemma \ref{ineq.cond} again, it is enough to show that
	\begin{align}\label{eq.mindeg3}
	\phi(p_1p_2\cdots p_{s-3}) (p_{s-1}p_{s}\cdots p_{r}-1) > \phi(p_{s-2}) \left[p_1p_2\cdots p_{s-3}  -  \phi(p_1p_2\cdots p_{s-3})\right].
	\end{align}
Since $\phi(p_{s-2})< p_{s-1}$, we have	
	\begin{align}\label{eq.mindeg4}
	\phi(p_1p_2\cdots p_{s-2}) (p_{s}p_{s+1}\cdots p_{r}-1) & < \phi(p_1p_2\cdots p_{s-3})p_{s-1} (p_{s}p_{s+1}\cdots p_{r}-1) \nonumber\\
	& < \phi(p_1p_2\cdots p_{s-3}) (p_{s-1}p_s\cdots p_r- 1).
	\end{align}	
Moreover,	
	\begin{align}\label{eq.mindeg5}
   \phi(p_{s-1}) \left[p_1p_2\cdots p_{s-2} - \phi(p_1p_2\cdots p_{s-2})\right]
    &>  p_1p_2\cdots p_{s-3}p_{s-2} - \phi(p_1p_2\cdots p_{s-2}) \nonumber \\
    & =  p_1p_2\cdots p_{s-3} \left[\phi(p_{s-2})+1\right] - \phi(p_1p_2\cdots p_{s-2}) \nonumber \\
    & >  \phi(p_{s-2}) \left[p_1p_2\cdots p_{s-3} - \phi(p_1p_2\cdots p_{s-3})\right].
	\end{align}
Now (\ref{eq.mindeg3}) follows from the inequalities (\ref{eq.mindeg2}), (\ref{eq.mindeg4}) and (\ref{eq.mindeg5}).
\end{proof}
\medskip

\begin{proof}[{\bf Proof of Theorem \ref{mindeg.main}}]
If $r=3$, then $\delta(\mathcal{P}(C_n)) = \min \{\deg(p_2p_3),\deg({p_3})\}$ by Lemma \ref{pp2}. Assume that $r\geq 4$. By Lemma \ref{prime.ineq}(i), we have $(r-2)\phi(p_1p_2\cdots p_{r-3})\geq p_1p_2\cdots p_{r-3}$, and this gives
$$r-3\geq \frac{p_1p_2\cdots p_{r-3}}{\phi(p_1p_2\cdots p_{r-3})} - 1.$$
Then
$$p_{r-1}p_r >(r-2)\phi(p_{r-2})\geq (r-3)\phi(p_{r-2}) +1\geq \left( \dfrac{p_1p_2\cdots p_{r-3}}{\phi(p_1p_2\cdots p_{r-3})} - 1\right) \phi(p_{r-2}) +1.$$ \\
So inequality (\ref{prod.p.ineq}) is satisfied with $s=r-1$ and hence $\deg({p_{r-2}p_{r-1}p_r}) > \deg({p_{r-1}p_r})$ by Lemma \ref{ineq.cond}. Then, using Lemma \ref{higher-lower} repeatedly, we have
$$\deg(p_2p_3\cdots p_r)>\deg(p_3p_4\cdots p_r)>\cdots >\deg({p_{r-2}p_{r-1}p_r}) > \deg({p_{r-1}p_r}).$$
Therefore, by Lemma \ref{pp2}, we get
$$\delta(\mathcal{P}(C_n)) = \min\{\deg(p_{r-1}p_r), \deg(p_r)\}.$$
By Lemma \ref{ineq.cond}, $\delta(\mathcal{P}(C_n)) = \deg(p_r)$ if and only if $\phi(p_r) \geq \left( \dfrac{p_1p_2\cdots p_{r-2}}{\phi(p_1p_2\cdots p_{r-2})} - 1\right) \phi(p_{r-1})$.

Now suppose that $\phi(p_{r}) \geq (r-2) \phi(p_{r-1})$. Since $(r-1)\phi(p_1p_2\cdots p_{r-2})\geq p_1p_2\cdots p_{r-2}$ by Lemma \ref{prime.ineq}(i), we have
$$r-2\geq \frac{p_1p_2\cdots p_{r-2}}{\phi(p_1p_2\cdots p_{r-2})} - 1.$$
Therefore,
$$p_r\geq (r-2) \phi(p_{r-1})+1\geq \left(\frac{p_1p_2\ldots p_{r-2}}{\phi(p_1p_2\ldots p_{r-2})}-1 \right) \phi(p_{r-1})+1.$$
Then $\deg(p_{r-1}p_r)\geq \deg (p_r)$ by Lemma \ref{ineq.cond} and hence $\delta(\mathcal{P}(C_n)) = \deg(p_r)$. This completes the proof of Theorem \ref{mindeg.main}.
\end{proof}

\begin{example}
The following examples shows that all possibilities can occur in Theorem \ref{mindeg.main} for the minimum degree of $\mathcal{P}(C_n)$.
\begin{enumerate}[(i)]
\item If $n = 2 \cdot 13 \cdot 17 \cdot 19$, then $\delta(\mathcal{P}(C_n))=\deg(17 \cdot 19)< \deg(19)$.
\item If $n = 2 \cdot 13 \cdot 17 \cdot 23$, then $\delta(\mathcal{P}(C_n))=\deg(23)<\deg(17\cdot 23)$.
\item If $n=2\cdot 5\cdot 13\cdot 19$, then $\delta(\mathcal{P}(C_n))=\deg(13\cdot 19)=\deg(19)$.
\end{enumerate}
Note that if $n=2\cdot 3\cdot p_3\cdot p_4$ with $p_4=2p_3 -1$, then $\deg(p_{3}p_4)=\deg(p_4)$.
\end{example}

\section{Proof of Thereom \ref{thm.mindeg} and Corolary \ref{exact.mindeg}}	\label{general-n}

In this section, we take $n=p_1^{\alpha_1}p_2^{\alpha_2} \cdots p_r^{\alpha_r}$, where $r \geq 2$, $\alpha_1,\alpha_2,\ldots, \alpha_r$ are positive integers and $p_1,p_2,\ldots, p_r$ are prime numbers with $p_1<p_2<\cdots <p_r$.

\begin{proposition}\label{prop1}
Let $m=p_{k_1}^{\beta_{k_1}}p_{k_2}^{\beta_{k_2}}\cdots p_{k_s}^{\beta_{k_s}}$, where $2 \leq s \leq r$, $k_1<k_2<\cdots < k_s$ and $1\leq \beta_{k_i}\leq \alpha_{k_i}$ for $1\leq i\leq s$.
Suppose that one of the following two conditions holds:
\begin{enumerate}[\rm(i)]
\item $2\phi(p_{1} p_2\cdots p_{r}) \geq p_1p_2\cdots p_r$,
\item $\phi(p_{j+1}) \geq r \phi(p_j)$ for each $j\in \{1,2,\ldots, r-1\}$.
\end{enumerate}
Then $\deg(m) > \deg\left(\dfrac{m}{p_{k_i}} \right)$ for $i\in\{1,2,\ldots, s-1\}$ in $\mathcal{P}(C_n)$.
\end{proposition}

\begin{proof}
Fix $i\in\{1,2,\ldots, s-1\}$. Using the degree formula (\ref{eqn-2}), we have
$$\xi:= \deg(m) - \deg\left(\frac{m}{p_{k_i}} \right) = \frac{n}{m} - \frac{np_{k_i}}{m} +\theta =\theta - \frac{n}{m} \phi(p_{k_i}),$$
where
$$\theta:= \sum_{\substack{d | m}} \phi \left( \frac{n}{d}\right)  - \sum_{\substack{d \big\vert \frac{m}{p_{k_i}}}} \phi \left( \frac{n}{d}\right)+\phi\left(\dfrac{np_{k_i}}{m} \right) - \phi\left(\dfrac{n}{m}\right).$$
First calculate $\sum_{\substack{d | m}} \phi \left( \frac{n}{d}\right)  - \sum_{\substack{d \big\vert \frac{m}{p_{k_i}}}} \phi \left( \frac{n}{d}\right)$. Define $ n':= \dfrac{n}{p_{k_1}^{\alpha_{k_1}}p_{k_2}^{\alpha_{k_2}}\cdots p_{k_s}^{\alpha_{k_s}}} \times {p_{k_i}^{\alpha_{k_i} - \beta_{k_i}}}$. Then
$$ \sum_{d | m} \phi \left( \frac{n}{d}\right)  - \sum_{d \big\vert \frac{m}{p_{k_i}}} \phi \left( \frac{n}{d}\right) =  \sum_{d \big\vert \frac{m}{p_{k_i}^{\beta_{k_i}}}} \phi \left( \frac{n}{dp_{k_i}^{\beta_{k_i}}}\right)
 = \phi \left( n'\right) \times \left(\sum_{d \big\vert \frac{m}{p_{k_i}^{\beta_{k_i}}}} \phi \left( \frac{p_{k_1}^{\alpha_{k_1}}p_{k_2}^{\alpha_{k_2}}\cdots p_{k_s}^{\alpha_{k_s}}}{dp_{k_i}^{\alpha_{k_i}}}\right)\right).$$
Taking $x=\dfrac{p_{k_1}^{\alpha_{k_1}}p_{k_2}^{\alpha_{k_2}}\cdots p_{k_s}^{\alpha_{k_s}}}{p_{k_i}^{\alpha_{k_i}}}$ and $y=\dfrac{p_{k_1}^{\beta_{k_1}}p_{k_2}^{\beta_{k_2}}\cdots p_{k_s}^{\beta_{k_s}}}{p_{k_i}^{\beta_{k_i}}}=\dfrac{m}{p_{k_i}^{\beta_{k_i}}}$ in Lemma \ref{phi.sum}, we get
\begin{align}\label{align-num-1}
\sum_{d | m} \phi \left( \frac{n}{d}\right)  - \sum_{d \big\vert \frac{m}{p_{k_i}}} \phi \left( \frac{n}{d}\right)
 & =\phi \left( n'\right) \left(\prod^{s}_{\substack{j=1\\ j\neq i\\ \alpha_{k_j} = \beta_{k_j}}} p_{k_j}^{\alpha_{k_j}} \right)  \left(\prod^{s}_{\substack{j=1\\ j\neq i\\ \alpha_{k_j} > \beta_{k_j}}} \left( p_{k_j}^{\alpha_{k_j}} - p_{k_j}^{\alpha_{k_j} - \beta_{k_j} -1} \right) \right) \nonumber \\
 & = \phi \left( n'\right) \left(\prod^{s}_{\substack{j=1\\ j\neq i\\ \alpha_{k_j} = \beta_{k_j}}} p_{k_j}^{\alpha_{k_j}} \right) \left(\prod^{s}_{\substack{j=1\\ j\neq i\\ \alpha_{k_j} > \beta_{k_j}}}  p_{k_j}^{\alpha_{k_j} - \beta_{k_j} -1} \left(p_{k_j}^{\beta_{k_j} +1} -1\right) \right).
\end{align}
Next calculate $\phi \left( \dfrac{np_{k_i}}{m}\right) - \phi \left(\dfrac{n}{m}\right)$. We have
\begin{align} \label{align-num-2}
\phi \left( \frac{np_{k_i}}{m}\right) - \phi \left(\frac{n}{m}\right) & = \phi \left( \frac{n}{p_{k_1}^{\alpha_{k_1}}\cdots p_{k_s}^{\alpha_{k_s}}}\right) \phi \left( \frac{p_{k_1}^{\alpha_{k_1} - \beta_{k_1}} \cdots p_{k_s}^{\alpha_{k_s} - \beta_{k_s}}}{p_{k_i}^{\alpha_{k_i}-\beta_{k_i}}}\right)\nonumber \\
 & \qquad \times \left[\phi\left(p_{k_i}^{\alpha_{k_i}-\beta_{k_i}+1}\right)-\phi\left(p_{k_i}^{\alpha_{k_i}-\beta_{k_i}}\right)\right] \nonumber \\
 & \geq \phi \left( \frac{n}{p_{k_1}^{\alpha_{k_1}}\cdots p_{k_s}^{\alpha_{k_s}}}\right) \phi \left( \frac{p_{k_1}^{\alpha_{k_1} - \beta_{k_1}} \cdots p_{k_s}^{\alpha_{k_s} - \beta_{k_s}}}{p_{k_i}^{\alpha_{k_i}-\beta_{k_i}}}\right) \phi\left(p_{k_i}^{\alpha_{k_i}-\beta_{k_i}}\right) \left(p_{k_i}-2\right) \nonumber \\
& = \phi\left(n'\right) \phi \left( \frac{p_{k_1}^{\alpha_{k_1} - \beta_{k_1}} \cdots p_{k_s}^{\alpha_{k_s} - \beta_{k_s}}}{p_{k_i}^{\alpha_{k_i}-\beta_{k_i}}}\right) \left(p_{k_i}-2\right) \nonumber \\
& =\phi\left(n'\right) \left(\prod^{s}_{\substack{j=1\\ j\neq i\\ \alpha_{k_j} > \beta_{k_j}}}  p_{k_j}^{\alpha_{k_j} - \beta_{k_j} -1}\phi\left(p_{k_j}\right)\right) \left(p_{k_i}-2\right),
\end{align}
where equality in the above holds if and only if $\alpha_{k_i}=\beta_{k_i}$. Using (\ref{align-num-1}) and (\ref{align-num-2}), we get
\begin{align}
\theta & \geq \phi \left( n'\right) \left(\prod^{s}_{\substack{j=1\\ j\neq i\\ \alpha_{k_j} > \beta_{k_j}}}  p_{k_j}^{\alpha_{k_j} - \beta_{k_j} -1}\right) \nonumber\\
& \qquad \times \left[\left(\prod^{s}_{\substack{j=1\\ j\neq i\\ \alpha_{k_j} = \beta_{k_j}}} p_{k_j}^{\alpha_{k_j}} \right)  \left(\prod^{s}_{\substack{j=1\\ j\neq i\\ \alpha_{k_j} > \beta_{k_j}}} \left(p_{k_j}^{\beta_{k_j} +1} -1\right) \right)  + \left( \prod^{s}_{\substack{j=1\\ j\neq i\\ \alpha_{k_j} > \beta_{k_j}}}  \phi\left(p_{k_j}\right)\right) \left(p_{k_i}-2\right) \right] \label{align-num-3}\\
& = \phi \left( n'\right) \left(\prod^{s}_{\substack{j=1\\ j\neq i\\ \alpha_{k_j} > \beta_{k_j}}}  p_{k_j}^{\alpha_{k_j} - \beta_{k_j} -1}\right)\times \left( \prod^{s}_{\substack{j=1\\ j\neq i\\ \alpha_{k_j} > \beta_{k_j}}}  \phi\left(p_{k_j}\right)\right) \nonumber\\
& \qquad \times \left[\left(\prod^{s}_{\substack{j=1\\ j\neq i\\ \alpha_{k_j} = \beta_{k_j}}} p_{k_j}^{\alpha_{k_j}} \right)  \left(\prod^{s}_{\substack{j=1\\ j\neq i\\ \alpha_{k_j} > \beta_{k_j}}} \left(p_{k_j}^{\beta_{k_j}} +\cdots + p_{k_j} + 1\right) \right)  + \left(p_{k_i}-2\right) \right] \nonumber\\
& = \phi \left( n'\right) \left(\prod^{s}_{\substack{j=1\\ j\neq i\\ \alpha_{k_j} > \beta_{k_j}}}  \phi\left(p_{k_j}^{\alpha_{k_j} - \beta_{k_j}}\right)\right) \nonumber \\
& \qquad \times \left[  \left(\prod^{s}_{\substack{j=1\\ j\neq i\\ \alpha_{k_j} = \beta_{k_j}}} p_{k_j}^{\alpha_{k_j}} \right)  \left(\prod^{s}_{\substack{j=1\\ j\neq i\\ \alpha_{k_j} > \beta_{k_j}}} \left(p_{k_j}^{\beta_{k_j}} +\cdots + p_{k_j} + 1\right) \right) + \phi(p_{k_i})-1 \right] \nonumber \\
& \geq \phi \left( n'\right) \left(\prod^{s}_{\substack{j=1\\ j\neq i\\ \alpha_{k_j} > \beta_{k_j}}}  \phi\left(p_{k_j}^{\alpha_{k_j} - \beta_{k_j}}\right)\right) \times \left[\phi\left(p_{k_s}\right) + \phi\left(p_{k_i}\right)\right] \nonumber\\
& \geq \frac{1}{m}\times \left[p_1^{\alpha_1-1}p_2^{\alpha_2-1}\cdots p_r^{\alpha_r-1} \phi(p_1p_2\cdots p_r) \times \left[\phi\left(p_{k_s}\right) + \phi\left(p_{k_i}\right)\right] \right] . \label{align-num-4}
\end{align}
Note that equality holds in (\ref{align-num-3}) if and only if $\alpha_{k_i}=\beta_{k_i}$, which follows from (\ref{align-num-2}). It can be seen that equality holds in (\ref{align-num-4}) if and only if $\alpha_{k_j}>\beta_{k_j}$ for all $j\in\{1,2,\ldots, s\}$. Combining these two facts, we thus have
\begin{align}
\theta & >  \frac{1}{m}\times \left[p_1^{\alpha_1-1}p_2^{\alpha_2-1}\cdots p_r^{\alpha_r-1} \phi(p_1p_2\cdots p_r) \times \left[\phi\left(p_{k_s}\right) + \phi\left(p_{k_i}\right)\right] \right].\nonumber
\end{align}
Finally, we get
\begin{align}\label{align-num-6}
\xi & > \frac{1}{m}\left[p_1^{\alpha_1-1}\cdots p_r^{\alpha_r-1} \phi(p_1p_2\cdots p_r) \times \left[\phi\left(p_{k_s}\right) + \phi\left(p_{k_i}\right)\right] \right] - \frac{n}{m} \phi(p_{k_i}) \nonumber \\
& = \frac{p_1^{\alpha_1-1}\cdots p_r^{\alpha_r-1}}{m} \left[\left[\phi\left(p_{k_s}\right) + \phi\left(p_{k_i}\right)\right] \phi(p_1p_2\ldots p_r) - \phi\left(p_{k_i}\right) p_1p_2\cdots p_r \right].
\end{align}
Since $k_s >k_i$, we have $p_{k_s}>p_{k_i}$. So $\phi\left(p_{k_s}\right) > \phi\left(p_{k_i}\right)$ and hence $\phi\left(p_{k_s}\right) + \phi\left(p_{k_i}\right)>2\phi\left(p_{k_i}\right)$. If $2\phi(p_{1} p_2\cdots p_{r}) \geq p_1p_2\cdots p_r$, then it follows from (\ref{align-num-6}) that $\deg(m) - \deg\left(\frac{m}{p_{k_i}} \right)>0$.
If $\phi(p_{j+1}) \geq r \phi(p_j)$ for each $j\in \{1,2,\ldots, r-1\}$, then $\phi(p_{k_s})\geq r\phi(p_{k_i})$ and so $\phi\left(p_{k_s}\right) + \phi\left(p_{k_i}\right)\geq (r+1)\phi\left(p_{k_i}\right)$. Using Lemma \ref{prime.ineq}(i), it again follows from (\ref{align-num-6}) that $\deg(m) - \deg\left(\frac{m}{p_{k_i}} \right)>0$. This completes the proof.
\end{proof}

\begin{proof}[{\bf Proof of Theorem \ref{thm.mindeg}}]
Let $m$ be a proper divisor of $n$. We can write $m=p_{k_1}^{\beta_{k_1}}p_{k_2}^{\beta_{k_2}}\cdots p_{k_s}^{\beta_{k_s}}$ for some $s\in\{1,2,\ldots,r\}$, where $k_1<k_2<\cdots < k_s$ and $1\leq \beta_{k_i}\leq \alpha_{k_i}$ for $1\leq i\leq s$. If $s=1$, then $\deg(m)=\deg\left(p_{k_1}^{\beta_{k_1}}\right)\geq \deg\left(p_{k_1}^{\alpha_{k_1}}\right)$ by Proposition \ref{degcompare}(ii). So assume that $s\geq 2$. Then applying Proposition \ref{prop1} repeatedly, we find that $\deg(m)>\deg\left(p_{k_s}^{\beta_{k_s}}\right)\geq \deg\left(p_{k_s}^{\alpha_{k_s}}\right)$. Here the last inequality holds again by Proposition \ref{degcompare}(ii). Thus
$$\delta(\mathcal{P}(C_n))=\min\{\deg\left(p_{i}^{\alpha_{i}}\right): 1\leq i\leq r\}.$$
By Proposition \ref{degcompare}(i), we have $\deg\left(p_{1}^{\alpha_{1}}\right)> \deg\left(p_{r}^{\alpha_{r}}\right)$. Let $t\in\{2,3,\ldots, r\}$ be the largest integer such that $\alpha_t\geq \alpha_j$ for $2\leq j\leq r$. Then, by Proposition \ref{degcompare}(ii) and (iii), we have $\deg\left(p_{j}^{\alpha_{j}}\right)> \deg\left(p_{t}^{\alpha_{t}}\right)$ for $2\leq j\leq t-1$ (if $t\geq 3$). It now follows that
$$\delta(\mathcal{P}(C_n)) = \min\{\deg\left(p_s^{\alpha_s}\right):t \leq s \leq r\}.$$
This completes the proof.	
\end{proof}

\begin{example}
Let $n=2\cdot 3 \cdot 5\cdot 11$. Then the minimum degree of $\mathcal{P}(C_n)$ is equal to $\deg(11)$ by Theorem \ref{mindeg.main}, but none of the two conditions mentioned in Theorem \ref{thm.mindeg} is satisfied. Thus each of the two conditions stipulated in  Theorem \ref{thm.mindeg} is sufficient but not necessary.
\end{example}

\begin{example}
Let $n=2^2\cdot 7 \cdot 11 \cdot 13$. By Proposition \ref{degcompare}(iii), we have $\deg(13)<\deg(11)<\deg(7)$. Using the degree formula (\ref{eqn-2}), it can be seen that $\deg(11\cdot 13)<\deg(13)$. This shows that if none of the two conditions stated in Theorem \ref{thm.mindeg} is satisfied, then the minimum of degree of $\mathcal{P}(C_n)$ may not be equal to the degree of $p_i^{\alpha_i}$ for any $i\in\{2,3,\ldots,r\}$.
\end{example}

\begin{proof}[{\bf Proof of Corollary \ref{exact.mindeg}}]
If $p_1\geq r+1$, then $2\phi(p_1p_2\cdots p_r)\geq p_1p_2\cdots p_r$ by Lemma \ref{prime.ineq}(ii). If $p_{i+1} > rp_{i}$ for each $i\in\{1,2,\ldots,r-1\}$, then $\phi(p_{i+1})=p_{i+1} -1 > rp_{i} -1> rp_{i}-r= r \phi(p_i)$ for $1\leq i\leq r-1$. So, by Theorem \ref{thm.mindeg}, the minimum degree of  $\mathcal{P}(C_n)$ is equal to $\deg\left({p_i^{\alpha_i}}\right)$ for some $i\in\{2,3,\ldots,r\}$ and the result follows for $r=2$.
Assume that $r\geq 3$. For $2\leq i < r$, let $m=\dfrac{n}{p_i^{\alpha_i}p_r^{\alpha_r}}$.
By Lemma \ref{prime.ineq}(ii), $\phi\left(\dfrac{p_1 p_2\cdots p_r}{p_i} \right)\geq \dfrac{p_1p_2 \cdots p_r}{rp_i}$ and so
\begin{align}\label{align-num-7}
p_i^{\alpha_i-1} \left[(p_r-1)\phi(m) - p_im\right] & = p_1^{\alpha_1-1} \cdots p_{r-1}^{\alpha_{r-1}-1} \left[\phi\left(\frac{p_1 p_2\cdots p_r}{p_i} \right) - {p_1 p_2\cdots p_{r-1}} \right] \nonumber\\
& \geq p_1^{\alpha_1-1} \cdots p_{r-1}^{\alpha_{r-1}-1} \left[\frac{p_1p_2 \cdots p_r}{rp_i}  - {p_1p_2 \cdots p_{r-1}} \right] \nonumber \\
& = \frac{ p_1^{\alpha_1} \cdots p_{r-1}^{\alpha_{r-1}}}{rp_i} (p_r - rp_i).
\end{align}
From the given conditions in both cases, we have $p_r> rp_{r-1}\geq rp_i$ and then the result follows from (\ref{eqn-5}) and (\ref{align-num-7}).
\end{proof}

\section{Proof of Theorem \ref{mindeg.3prime}}\label{three-primes}

In this section, take $n=p_1^{\alpha_1}p_2^{\alpha_2}p_3^{\alpha_3}$, where $\alpha_1,\alpha_2,\alpha_3$ are positive integers and $p_1,p_2,p_3$ are prime numbers with $p_1<p_2<p_3$.

\begin{lemma}\label{lem.beta2}
Let $i,j\in\{1,2,3\}$ with $i< j$. If $1\leq \beta_i\leq\alpha_i$, then $\deg\left(p_{i}^{\beta_i}p_{j}^{\beta_j}\right) > \deg\left( p_{i}^{\beta_i-1}p_{j}^{\beta_j} \right)$ for $\beta_j \geq 2$.
\end{lemma}

\begin{proof}
Using the degree formula (\ref{eqn-2}), we have
\begin{align*}
\deg\left(p_{i}^{\beta_i}p_{j}^{\beta_j}\right) - \deg\left( p_{i}^{\beta_i-1}p_{j}^{\beta_j} \right) & = \frac{n}{p_i^{\beta_i}p_j^{\beta_j}}-\frac{n}{p_i^{\beta_i-1}p_j^{\beta_j}} + \sum_{\substack{d \big\vert p_i^{\beta_i}p_j^{\beta_j}}} \phi \left( \frac{n}{d}\right) - \sum_{\substack{d \big\vert p_i^{\beta_i-1}p_j^{\beta_j}}} \phi \left( \frac{n}{d}\right)\\
& \qquad + \phi\left(\frac{n}{p_i^{\beta_i-1}p_j^{\beta_j}} \right)  - \phi\left(\frac{n}{p_i^{\beta_i}p_j^{\beta_j}} \right).
\end{align*}
Let $\{k\}=\{1,2,3\}\setminus \{i,j\}$. Then
\begin{equation}\label{eqn-6}
\frac{n}{p_i^{\beta_i}p_j^{\beta_j}}-\frac{n}{p_i^{\beta_i-1}p_j^{\beta_j}}= - p_{i}^{\alpha_{i} - \beta_i} p_{j}^{\alpha_{j} - \beta_j} p_{k}^{\alpha_{k}} \phi(p_{i})
\end{equation}
and
\begin{equation}\label{eqn-7}
\sum_{\substack{d \big\vert p_i^{\beta_i}p_j^{\beta_j}}} \phi \left( \frac{n}{d}\right) - \sum_{\substack{d \big\vert p_i^{\beta_i-1}p_j^{\beta_j}}} \phi \left( \frac{n}{d}\right)
\geq \phi \left({p_{i}^{\alpha_{i} - \beta_i}} p_{k}^{\alpha_{k}}\right) \left( p_{j}^{\alpha_j} - p_{j}^{\alpha_j-\beta_j-1}  \right),
\end{equation}
where equality holds if and only if $\alpha_j > \beta_j$. We also have
\begin{equation}\label{eqn-8}
\phi\left(\frac{n}{p_i^{\beta_i-1}p_j^{\beta_j}} \right)  - \phi\left(\frac{n}{p_i^{\beta_i}p_j^{\beta_j}} \right)\geq \phi \left({p_{i}^{\alpha_{i} - \beta_i}} p_{k}^{\alpha_{k}}\right)  \phi\left(p_{j}^{\alpha_{j} - \beta_j}\right) ( \phi(p_{i}) -1 ),
\end{equation}
where equality holds if and only if $\alpha_i = \beta_i$. From (\ref{eqn-6}), (\ref{eqn-7}) and (\ref{eqn-8}), we get
\begin{align*}
&\deg\left(p_{i}^{\beta_i}p_{j}^{\beta_j}\right) - \deg\left( p_{i}^{\beta_i-1}p_{j}^{\beta_j} \right)\\
& \geq \phi \left({p_{i}^{\alpha_{i} - \beta_i}} p_{k}^{\alpha_{k}}\right) \left[\left( p_{j}^{\alpha_j} - p_{j}^{\alpha_j-\beta_j-1}  \right) +  \phi\left(p_{j}^{\alpha_{j} - \beta_j}\right) ( \phi(p_{i}) -1 ) \right]  - p_{i}^{\alpha_{i} - \beta_i} p_{j}^{\alpha_{j} - \beta_j} p_{k}^{\alpha_{k}} \phi(p_{i}) \\
& \geq p_{i}^{\alpha_{i} - \beta_i-1} p_{j}^{\alpha_{j} - \beta_j-1} p_{k}^{\alpha_{k}-1} \left[\phi\left( p_{i} p_{k} \right) \left[ \left( p_{j}^{\beta_j+1} -1 \right )  +  \phi\left(p_{j}\right) ( \phi(p_{i}) -1 ) \right] - p_{i} p_{j} p_{k} \phi(p_{i}) \right]  \\
& = p_{i}^{\alpha_{i} - \beta_i-1} p_{j}^{\alpha_{j} - \beta_j-1} p_{k}^{\alpha_{k}-1} \left[\phi\left( p_{i} p_{j} p_{k} \right) \left( p_{j}^{\beta_j} + \ldots + p_{j} + \phi(p_{i}) \right ) - p_{i} p_{j} p_{k} \phi(p_{i}) \right]  \\
& = p_{i}^{\alpha_{i} - \beta_i-1} p_{j}^{\alpha_{j} - \beta_j-1} p_{k}^{\alpha_{k}-1} \phi(p_{i}) \left[ \phi\left(p_{j} p_{k} \right) \left( p_{j}^{\beta_j} + \ldots + p_{j} + \phi(p_{i}) \right ) - p_{i} p_{j} p_{k}  \right].
\end{align*}
Since $\beta_j \geq 2$ and $j>i$ by the given hypotheses, we have $ p_{j}^{\beta_j}+\cdots + p_{j} + \phi(p_{i})  > 3 p_i$. So
$\phi\left(p_{j} p_{k} \right) \left( p_{j}^{\beta_j} + \ldots + p_{j} + \phi(p_{i}) \right ) - p_{i} p_{j} p_{k}> 3p_i\phi\left(p_{j} p_{k} \right)- p_{i} p_{j} p_{k}\geq 0,$
where the last inequality follows using Lemma \ref{prime.ineq}(i). Hence the lemma follows.
\end{proof}

\begin{lemma}\label{lem1}
Let $\{i,j,k\}=\{1,2,3\}$ with $i< j$. If $(p_{i}+ p_{j})\phi \left( p_{j}p_k \right) -  p_ip_{j}p_k>0$, then  $\deg\left(p_{i}^{\beta_i}p_{j}\right) > \deg\left(p_{j} \right)$ for $1\leq\beta_i\leq\alpha_i$.
\end{lemma}

\begin{proof}
Using the degree formula (\ref{eqn-2}), we get
\begin{align*}
& \deg\left(p_{i}^{\beta_i}p_{j}\right) - \deg\left(p_{j} \right) = \frac{n}{p_i^{\beta_i}p_j}-\frac{n}{p_j} + \sum_{\substack{d | p_i^{\beta_i}p_j}} \phi \left( \frac{n}{d}\right) - \sum_{\substack{d | p_j}} \phi \left( \frac{n}{d}\right) + \phi\left(\frac{n}{p_j} \right)  - \phi\left(\frac{n}{p_i^{\beta_i}p_j} \right).
\end{align*}
We have
\begin{equation}\label{eqn-9}
\frac{n}{p_i^{\beta_i}p_j}-\frac{n}{p_j}= -  p_{j}^{\alpha_{j} - 1}p_k^{\alpha_k}  \left(p_{i}^{\alpha_i} -p_{i}^{\alpha_{i} - \beta_i}\right)
\end{equation}
and
\begin{align}\label{eqn-10}
\sum_{\substack{d | p_i^{\beta_i}p_j}} \phi \left( \frac{n}{d}\right) - \sum_{\substack{d | p_j}} \phi \left( \frac{n}{d}\right)& =\phi \left(p_k^{\alpha_k}\right)  \left(\sum_{l=1}^{\beta_i} \phi \left(  {p_{i}^{\alpha_{i} - l}} \right) \right) \left(\sum_{l=0}^{1} \phi \left(  {p_{j}^{\alpha_{j} - l}} \right) \right )\nonumber\\
& \geq \phi \left(p_k^{\alpha_k}\right)  \left(\sum_{l=1}^{\beta_i} \phi \left(  {p_{i}^{\alpha_{i} - l}} \right) \right) \left ( p_{j}^{\alpha_{j}} - p_{j}^{\alpha_{j} - 2} \right ),
\end{align}
where equality holds if and only if $\alpha_j>1$. We also have
\begin{equation}\label{eqn-11}
\phi\left(\frac{n}{p_j} \right)  - \phi\left(\frac{n}{p_i^{\beta_i}p_j} \right)=\phi \left( p_{j}^{\alpha_{j} - 1}p_k^{\alpha_k} \right) \left(\phi\left(p_{i}^{\alpha_i}\right) - \phi\left(p_{i}^{\alpha_{i} - \beta_i}\right)\right).
\end{equation}
From (\ref{eqn-9}), (\ref{eqn-10}) and (\ref{eqn-11}), we get
\begin{align*}
\deg\left(p_{i}^{\beta_i}p_{j}\right) - \deg\left(p_{j} \right) & \geq \phi \left(p_k^{\alpha_k}\right) \left(\sum_{l=1}^{\beta_i} \phi \left(  {p_{i}^{\alpha_{i} - l}} \right) \right) \left ( p_{j}^{\alpha_{j}} - p_{j}^{\alpha_{j} - 2} \right )\\
& \quad + \phi \left( p_{j}^{\alpha_{j} - 1}p_k^{\alpha_k} \right) \left(\phi\left(p_{i}^{\alpha_i}\right) - \phi\left(p_{i}^{\alpha_{i} - \beta_i}\right)\right)  -  p_{j}^{\alpha_{j} - 1}p_k^{\alpha_k}  \left(p_{i}^{\alpha_i} -p_{i}^{\alpha_{i} - \beta_i}\right)  \\
& \geq p_{j}^{\alpha_{j} - 2} p_k^{\alpha_k-1} \Bigg[ \left ( p_{j}^{2} - 1 \right )\phi \left(p_k\right)  \left(\sum_{l=1}^{\beta_i} \phi \left(  {p_{i}^{\alpha_{i} - l}} \right) \right)\\
& \qquad\qquad + \phi \left( p_{j}p_k \right) \left(\phi\left(p_{i}^{\alpha_i}\right) - \phi\left(p_{i}^{\alpha_{i} - \beta_i}\right)\right)  -  p_{j}p_k  \left(p_{i}^{\alpha_i} -p_{i}^{\alpha_{i} - \beta_i}\right) \Bigg]\\
& = p_{j}^{\alpha_{j} - 2} p_k^{\alpha_k-1} \Bigg[  \phi \left( p_{j}p_k \right) \Bigg[\left ( p_{j}+ 1 \right ) \left(\sum_{l=1}^{\beta_i} \phi \left(  {p_{i}^{\alpha_{i} - l}} \right) \right) \\
& \qquad\qquad\qquad\qquad + \phi\left(p_{i}^{\alpha_i}\right) - \phi\left(p_{i}^{\alpha_{i} - \beta_i}\right) \Bigg] -  p_{j}p_k  \left(p_{i}^{\alpha_i} -p_{i}^{\alpha_{i} - \beta_i}\right)\Bigg]\\
& \geq p_{j}^{\alpha_{j} - 2} p_k^{\alpha_k-1} \bigg[ \phi \left( p_{j}p_k \right) \left[ p_{j} \left ( p_{i}^{\alpha_i-1} -p_{i}^{\alpha_{i} - \beta_i-1} \right )  + p_{i}^{\alpha_i} -p_{i}^{\alpha_{i} - \beta_i} \right]\\
& \qquad\qquad\qquad\qquad\qquad\qquad\qquad\qquad\qquad - p_{j}p_k  \left(p_{i}^{\alpha_i} -p_{i}^{\alpha_{i} - \beta_i}\right)\bigg]\\
& = p_{j}^{\alpha_{j} - 2} p_k^{\alpha_k-1} \left ( p_{i}^{\alpha_i-1} -p_{i}^{\alpha_{i} - \beta_i-1} \right ) \left[(p_{i}   + p_{j}) \phi \left( p_{j}p_k \right) -  p_ip_{j}p_k\right].
\end{align*}
Since $(p_{i}+ p_{j})\phi \left( p_{j}p_k \right) -  p_ip_{j}p_k>0$, it follows from the above that $\deg\left(p_{i}^{\beta_i}p_{j}\right) > \deg\left(p_{j} \right)$.
\end{proof}

\begin{proof}[{\bf Proof of Theorem \ref{mindeg.3prime}}]
If $p_1 \geq 4$, then $2\phi(p_1 p_2 p_3) \geq p_1 p_2 p_3$ by Lemma \ref{prime.ineq}(ii). So $\delta(\mathcal{P}(C_n)) = \min\{\deg\left({p_2^{\alpha_2}}\right),\deg\left({p_3^{\alpha_3}}\right)\}$ by Theorem \ref{thm.mindeg}. Now assume that $p_1 = 2$ or $3$. In view of Proposition \ref{degcompare}(i), (ii) and (iv), the minimum degree of $\mathcal{P}(C_n)$ can be attained at the vertex $p_{2}^{\alpha_2}$ or $p_{3}^{\alpha_3}$, or at a vertex of the form $p_{i}^{\beta_i}p_{j}^{\beta_j}$ for some $i,j\in\{1,2,3\}$ with $i<j$, where $1\leq \beta_i \leq \alpha_i$ and $1\leq \beta_j \leq \alpha_j$.

Consider the vertices of the form $p_{i}^{\beta_i}p_{j}^{\beta_j}$ with $i<j$. We show that $\deg\left(p_{i}^{\beta_i}p_{j}^{\beta_j}\right) > \deg\left(p_{j}^{\beta_j} \right)$. Then Proposition \ref{degcompare}(ii) implies that $\deg\left(p_{j}^{\beta_j}\right)\geq \deg\left( p_j^{\alpha_j}\right)$ and this would complete the proof.

If $\beta_j\geq 2$, then applying Lemma \ref{lem.beta2} repeatedly we find that $\deg\left(p_{i}^{\beta_i}p_{j}^{\beta_j}\right)>\deg\left(p_j^{\beta_j}\right)$.
Suppose that $\beta_j=1$. Let $\{k\}=\{1,2,3\}\setminus\{i,j\}$. We show that
\begin{equation}\label{eqn-12}
(p_{i}+ p_{j})\phi \left( p_{j}p_k \right) -  p_ip_{j}p_k>0.
\end{equation}
Then Lemma \ref{lem1} implies that $\deg\left(p_{i}^{\beta_i}p_{j}\right) > \deg\left(p_{j} \right)$. Clearly, (\ref{eqn-12}) holds using Lemma \ref{prime.ineq}(i) if $p_{j} > 2p_i$. Since $p_j\neq 2p_i$, assume that $p_{j} < 2 p_i$. We have the following two cases.
\begin{enumerate}[$\bullet$]
\item $i=1$: Since $p_1\in\{2,3\}$ and $p_{j} < 2 p_1$, we have $j=2$ and $(p_1,p_2)=(2,3)$ or $(3,5)$.
If $(p_1, p_2) = (2,3)$, then $\left(p_{1} + p_{2}\right)\phi\left(p_{2}p_3\right) - p_{1} p_{2} p_3 = 10 \phi \left(p_3\right)- 6 p_3 = 4 p_3 -10 > 0$ as $p_3\geq 5$. If $(p_1, p_2)=(3,5)$, then $p_3\geq 7$ and $\left(p_{1} + p_{2}\right)\phi\left(p_{2}p_3\right)-  p_{1} p_{2} p_3 = 32 \phi\left(p_3\right)- 15 p_3 = 17p_3 -32 > 0$.

\item $(i,j)=(2,3)$: Here $k=1$ and $p_3\geq p_2 + 2$. If $p_k=p_1=2$, then $\left ( p_{2} + p_{3} \right)\phi \left(p_1p_{3}\right) -  p_{1} p_{2} p_3 = p_3^2 - \left ( p_{2} + p_{3} \right ) - p_{2} p_3 =p_3( p_3- p_{2} -1)-p_2 > 0$. If $p_k=p_1=3$, then $\left(p_{2}+ p_{3}\right)\phi \left(p_1p_{3}\right) -  p_{1} p_{2} p_3 = 2p_3^2 - 2\left ( p_{2} + p_{3} \right )- p_{2} p_3 =p_3( 2p_3- p_{2} -2)-2p_2 > 0.$
\end{enumerate}	
This completes the proof.
\end{proof}	

\begin{corollary}\label{coro}
If $p_3 \geq 2p_2+1$, then $\delta(\mathcal{P}(C_n)) = \deg\left(p_3^{\alpha_3}\right)$.
\end{corollary}

\begin{proof}
By (\ref{eqn-5}), we have $\deg\left({p_2^{\alpha_2}}\right) - \deg\left({p_3^{\alpha_3}}\right) > p_2^{\alpha_2-1} \left[(p_3-1)\phi \left(p_1^{\alpha_1}\right) - p_2p_1^{\alpha_1}\right]$.
Since $p_3\geq 2p_2+1$, it follows that $\deg\left({p_2^{\alpha_2}}\right) - \deg\left({p_3^{\alpha_3}}\right)>p_1^{\alpha_1-1}p_2^{\alpha_2} \left (2\phi(p_1) - p_1\right) \geq 0$ and so $\delta(\mathcal{P}(C_n)) = \deg\left(p_3^{\alpha_3}\right)$ by Theorem \ref{mindeg.3prime}.
\end{proof}

\begin{example}
Take $n= 2 \cdot  3^3 \cdot 5$. Then $\delta(\mathcal{P}(C_n))=\deg(3^3)=113 < 125 = \deg(5)$. It follows that if $p_3 < 2p_2+1$, then $\delta(\mathcal{P}(C_n)) = \deg\left(p_2^{\alpha_2}\right)< \deg\left(p_3^{\alpha_3}\right)$ may occur.
\end{example}

\vskip .3cm

\noindent\underline{\bf Addresses}:\\
{\bf Ramesh Prasad Panda, Kamal Lochan Patra, Binod Kumar Sahoo}
\begin{enumerate}
\item[1)] School of Mathematical Sciences, National Institute of Science Education and Research (NISER), Bhubaneswar, P.O.- Jatni, District- Khurda, Odisha - 752050, India.

\item[2)] Homi Bhabha National Institute (HBNI), Training School Complex, Anushakti Nagar, Mumbai - 400094, India.
\end{enumerate}

\noindent{\bf Emails}: rppanda@niser.ac.in, klpatra@niser.ac.in, bksahoo@niser.ac.in
\end{document}